\documentclass[]{amsart}
\usepackage[utf8]{inputenc}
\usepackage{amsfonts}
\usepackage{amsmath,amscd}
\usepackage{amsthm}
\usepackage{latexsym}
\usepackage{amssymb}
\usepackage{enumerate}
\usepackage{url}
\markleft{\hfill G. Ramírez-Raposo\hfill }
\newtheorem{theorem}{Theorem}
\newtheorem*{theorem*}{Theorem}

\newtheorem{definition}[theorem]{Definition}
\newtheorem{lemma}[theorem]{Lemma}
\newtheorem{remark}[theorem]{Remark}
\newtheorem*{claim*}{Claim}

\usepackage{mathdots}

\usepackage{geometry}
\usepackage{tikz}
\usetikzlibrary{arrows,matrix,positioning}

\usepackage{caption}
\usepackage{subcaption}
\usepackage{graphicx}

\newcommand{\Q}{\mathbb{Q}}

\newcommand{\ie}{{\it{i.$\,$e.\ }}}

\DeclareMathOperator{\sign}{sign} 
\DeclareMathOperator{\s}{s} 

\newcommand{\sing}{\left[\substack{ \\ \oplus}\right]}
\newcommand{\hdoub}{\left[\substack{ \\- \oplus}\right]}
\newcommand{\vdoub}{\left[\substack{ - \\ \oplus}\right]}
\newcommand{\iquad}{\left[\substack{+ - \\- \oplus}\right]}

\newcommand{\tquad}{\ensuremath{\left[\substack{\oplus - \\- +}\right]}}

\newcommand{\nhdoub}{\left[\substack{ \\ \ominus +}\right]}
\newcommand{\nvdoub}{\left[\substack{ \ominus \\ +}\right]}

\title{A Detailed Proof of Pohst's Inequality}
\author{Gabriel Raposo}
\address{Department of Statistics, University of California-Berkeley, \newline 367 Evans Hall, Berkeley, CA 94720, USA}
\email{raposo@berkeley.edu}

\begin{document}
\begin{abstract}

In 1977 Pohst conjectured a certain inequality for $n$ variables and give a computer-assisted proof for $n\leq 10$. We give a proof for all $n$ using a combinatorial argument. This inequality yields a better bound for the regulator in terms of the discriminant for totally real number fields.  

\end{abstract}

	\maketitle

\section{Introduction}

\hspace{15pt} In 1952 Remak \cite{Re} proved that the product $\prod_{1\leq i \leq j \leq n} \big|1-\prod_{k=i}^j x_k\big| $ is bounded above by $(n+1)^{(n+1)/2}$ when the variables $x_i$ are complex numbers with modulus $|x_i|\leq 1$. He used this inequality to obtain a lower bound for the regulator $R_k$ in terms of the discriminant $D_k$ for number fields $k$. 

In 1977 Pohst \cite{Po} proved for $n \leq 10$ that the same product is bounded above by $2^{\lfloor \frac{n+1}{2}\rfloor}$ when the variables are real with absolute value $|x_i|\leq 1$. This improved the regulator lower bound for totally real number fields. He gave a computer-assisted proof by using some elementary inequalities (see Lemma \ref{Pohst}) and factorizing the product for each possible combination of signs for the variables $x_i$. Here we prove it for all $n$.

\begin{theorem*}
Let $v=(x_1,\dots,x_n)\in [-1,1]^n$ and $f_n(v):=\prod_{1\leq i \leq j \leq n} (1-\prod_{k=i}^j x_k)$. Then $f_n(v) \leq 2^{\lfloor \frac{n+1}{2}\rfloor}$, where $\lfloor x \rfloor$ denotes the floor of the real number $x$. Moreover, the maximum is attained if and only if the coordinates of $v$ are $0$ or $-1$ and have the greatest possible number of $-1$'s   without two consecutive $-1$'s. 
\end{theorem*}

\noindent Pohst's motivation  \cite{Po} was to improve on Remak's inequality relating discriminant and regulator for primitive totally real fields  (fields with no subfields other than itself and $\Q$) \cite[(3.15)]{Fr} . Indeed, our theorem improves Remak's
$$
\log|D_k|\le m\log(m)+\sqrt{\gamma_{m-1}(m^3-m)/3}\,(\sqrt{m}R_k)^{1/(m-1)}
$$
to
$$
\log|D_k|\le \lfloor m/2\rfloor\log(4)+\sqrt{\gamma_{m-1}(m^3-m)/3}\,(\sqrt{m}R_k)^{1/(m-1)},
$$
where $m:=[k:\Q])$ and $\gamma_{m-1}$ is Hermite's constant in dimension $m-1$.
 
The idea of our proof of Pohst's inequality is to note that the problem is easily solved when each of the variables $x_i$ is negative. We then show that we can always bound the product $f_n(v)$ by this special case, \ie
$$
f_n(x_1,\dots,x_n) \leq f_n(-|x_1|,\dots,-|x_n|) \qquad\qquad\big((x_1,\dots,x_n)\in[-1,1]^n\big).
$$ To prove this we will use a variation of Pohst's elementary inequalities (Lemma \ref{NewPohst}) that will allow  us to exchange the signs of certain combinations of terms of the product $f_n(v)$. We then show that for any combination of signs for the variables $x_i$ we can find a factorization of $f_n(v)$ to which the new elementary inequalities apply (Lemma $\ref{partition}$).

There is   a proof of Pohst's inequality in \cite{Be}. However, the proof is incomplete so we give a new one here. Battistoni and Molteni \cite{BM} gave a solution to this same problem, using similar methods. This solutions was not known by the author and not yet published at the time of submission of this article.

\section{Proof of Theorem}

\begin{lemma}\label{Pohst}
(Pohst)
\begin{enumerate}
    \item If $a \in [-1,1]$, then $(1-a) \leq 2$.
    \item If $a \in [0,1]$ and $b \in [-1,0]$, then $(1-a)(1-ab) \leq 1$.
    \item If $a,b \in [-1,1]$, then $(1-a)(1-b)(1-ab) \leq 2$.
    \item If $a \in [0,1]$ and $b,c \in [-1,0]$, then $(1-a)(1-ab)(1-ac)(1-abc)\leq 1$.
\end{enumerate}
In  \textit{(3)} the maximum is attained if and only if $(a,b)=(0,-1)$ or  $(a,b)=(-1,0)$.
\end{lemma}
\begin{proof}
See \cite[p. 468]{Po}, or use undergraduate calculus.
\end{proof}

\begin{lemma}\label{NewPohst}
For $a,b,c \in [0,1]$ the following hold.
\begin{enumerate}
    \item $(1-a)\leq (1+a)$.
    \item $(1-a)(1+ab)\leq(1+a)(1-ab)$.
    \item $(1-a)(1+ab)(1+ac)(1-abc)\leq(1+a)(1-ab)(1-ac)(1+abc)$.
\end{enumerate}
\end{lemma}
\begin{proof}
For (2)  note that $(1-a)\leq(1-ab)$ and $(1+ab)\leq(1+a)$. For (3) note that, since $x(1-y)\leq 1-y$ for $x\leq 1$ and $y \leq 1$, then $x+y\leq1+xy$. From this we get
\begin{equation*}
(1+ab)(1+ac )  =  1+a(b+c)+a^2bc \leq   1+a(1+bc)+a^2bc 
 = (1+a)(1+abc),
\end{equation*}
which we multiply with the following inequality,
\begin{equation*}
(1-a)(1-abc)   =  1-a(1+bc)+a^2bc 
  \leq   1-a(b+c)+a^2bc 
  =   (1-ab)(1-ac).\qedhere
\end{equation*}
\end{proof}

Let $v=(x_1,\dots,x_n)\in [-1,1]^n$ and $a_{v(i,j)}:= 1-\prod_{k=i}^j x_k$. Note that the terms $a_{v(i,j)}$ can be easily ordered in a triangular way, as seen in figures \ref{figure1}, \ref{figure3} and \ref{figure4}. Next we treat the case where $x_k \in [-1,0]$ for all $k$.

\begin{lemma}\label{IdealCase}
If $v=(x_1,\dots,x_n)\in [-1,0]^n$, then $f_n(v) \leq 2^{\lfloor \frac{n+1}{2}\rfloor}$.  Moreover, the maximum is attained if and only if the coordinates of $v$ are $0$ or $-1$ and $v$ has the greatest possible number of $-1$'s without two consecutive $-1$'s.
\end{lemma}

\begin{proof}
We find a good factorization for the product $f_n(v)=\prod_{1 \leq i\leq j \leq n} a_{v(i,j)}$ by means of lemma \ref{Pohst}. Figure \ref{figure1} shows an illustrative example of how the partition is done. The case $n=1$ is obvious. If $n$ is even then
\begin{multline*}
    f_n(v)=\big( \prod_{k=1}^{n/2} a_{v(2k-1,2k-1)}a_{v(2k,2k)} a_{v(2k-1,2k)} \big) \cdot \\ \big( \prod_{\substack{1 \leq k < n/2\\ 0 \leq j < n/2-k}} a_{v(2k,2k+2j+1)} a_{v(2k-1,2k+2j+1)}  a_{v(2k,2k+2j+2)} a_{v(2k-1,2k+2j+2)} \big).
\end{multline*}
The first product corresponds to the terms at the base of the triangle while the second product corresponds to the rectangles in the rest of the triangle.  Using lemma \ref{Pohst} \textit{(3)} for the first product and lemma \ref{Pohst} \textit{(4)} for the second one we conclude the even case. If $n\geq 3$ is odd, note that
\begin{center}
    $f_n(v)=f_{n-1}(x_1,\dots,x_{n-1}) a_{v(n,n)} \prod_{k=1}^{(n-1)/2} a_{v(2k,n)}a_{v(2k-1,n)}$
\end{center}
and use lemma \ref{Pohst} \textit{(2)} in the last product. 

\begin{figure}[h!] 
    \centering
    \begin{tikzpicture}
        \matrix [matrix of math nodes] (m)
        {
            a_{v(1,n)} &a_{v(2,n)} &a_{v(3,n)} &a_{v(4,n)} &a_{v(5,n)} &a_{v(6,n)} &\hdots  &a_{v(n,n)}  \\
            \vdots& \vdots &\vdots &\vdots &\vdots &\vdots & \iddots & \\
            a_{v(1,6)} &a_{v(2,6)} &a_{v(3,6)} &a_{v(4,6)} &a_{v(5,6)} &a_{v(6,6)} & &  \\    
            a_{v(1,5)} &a_{v(2,5)} &a_{v(3,5)} &a_{v(4,5)} &a_{v(5,5)} & & & \\      
            a_{v(1,4)} &a_{v(2,4)} &a_{v(3,4)} &a_{v(4,4)} & & & & \\     
            a_{v(1,3)} &a_{v(2,3)} &a_{v(3,3)} & & & & & \\    
            a_{v(1,2)} &a_{v(2,2)} & & & & & & \\     
            a_{v(1,1)} & & & & & & & \\      
        };
        \draw[color=red] (m-1-1.south west) rectangle (m-1-2.north east);
        \draw[color=red] (m-1-3.south west) rectangle (m-1-4.north east);
        \draw[color=red] (m-1-5.south west) rectangle (m-1-6.north east);
        \draw[color=red] (m-1-8.south west) rectangle (m-1-8.north east);
        \draw[color=red] (m-4-1.south west) rectangle (m-3-2.north east);
        \draw[color=red] (m-4-3.south west) rectangle (m-3-4.north east);
        \draw[color=red] (m-6-1.south west) rectangle (m-5-2.north east);
        \draw[color=red] (m-8-1.south west) -- (m-8-1.south east) -- (m-7-2.south east) -- (m-7-2.north east) -- (m-7-1.north west) -- (m-8-1.south west);
        \draw[color=red] (m-6-3.south west) -- (m-6-3.south east) -- (m-5-4.south east) -- (m-5-4.north east) -- (m-5-3.north west) -- (m-6-3.south west);
        \draw[color=red] (m-4-5.south west) -- (m-4-5.south east) -- (m-3-6.south east) -- (m-3-6.north east) -- (m-3-5.north west) -- (m-4-5.south west);
\end{tikzpicture}
    \caption{Factorization of $f_n(v)$ when $n$ is odd.}
    \label{figure1}
\end{figure}
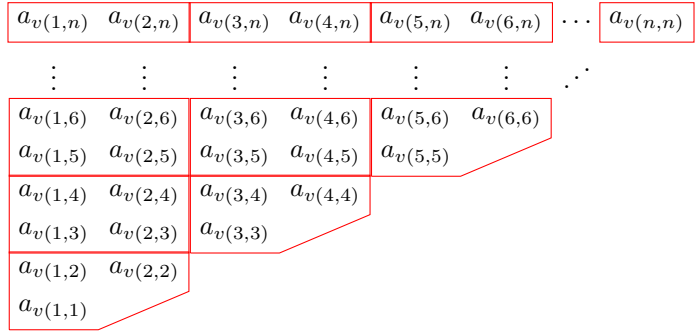

A necessary condition for the maximum to be attained is that each of the terms in the base of the triangle equals 2, i.e. $a_{v(2k-1,2k-1)}a_{v(2k,2k)} a_{v(2k-1,2k)}=2$, and $a_{v(n,n)}=2$ if $n$ is odd. Using the last statement in lemma \ref{Pohst}, this happens exactly when $x_{2k-1}=0$ and $x_{2k}=-1$ or $x_{2k-1}=-1$, and $x_{2k}=0$. However, we cannot have $x_{t}=x_{t+1}=-1$ for $t=1,\dots,n-1$ since this implies that $a_{v(t,t+1)}=0$ and then $f_n(v)=0$. If the coordinates of $v$ are $0$ or $-1$ and have the greatest number of $-1$'s but without two consecutive $-1$'s, then each each of the factors in the base of the triangle equal 2. Since for two consecutive coordinates of $v$ one of them is $0$, each factor $a_{v(i,j)}$ with $i\neq j$ is equal to 1. In consequence the product of the factors in the rectangles equals 1.
\end{proof}

\begin{definition}
Let $v=(x_1,\dots,x_n)\in \big([-1,1]\backslash\{0\}\big)^n$. Define the product sign   $\s$ of   $a_{v(i,j)}$ by $\s(a_{v(i,j)}):= \sign(1-a_{v(i,j)}) = \prod_{k=i}^j \sign(x_{k})$. We say that the term $a_{v(i,j)}=1-\prod_{k=i}^j x_k$ is non canonical if $\s(a_{v(i,j)})=(-1)^{i+j}$. We also define the set $J_v$ of non canonical indices of $v$, $J_v:=\{(i,j) : \s(a_{v(i,j)}) = (-1)^{i+j}\}$.
\end{definition}
\begin{remark}\label{Notethis} Note that if we let $-|v|:=(-|x_1|,\dots,-|x_n|)\in [-1,0)^n$ and if $a_{v(i,j)}$ is canonical, then $a_{v(i,j)}=a_{-|v|(i,j)}$ and $s(a_{-|v|(i,j)})=s(a_{v(i,j)})=(-1)^{i+j+1}$.
\end{remark}

We will partition  $J_v$ into  subsets of 1, 2 or 4 elements so that we can apply to the corresponding products  cases (1), (2) or (3) respectively   of lemma \ref{NewPohst}. For this we need the following definition. 

\begin{definition} \label{defpartition}
We say that $J_v$ has a good partition if there is a partition $\pi_v$ of $J_v$ such that if $p \in \pi_v$ then  one of the following holds.

\noindent (1) $p=\{(i,j)\}$ with $\s(a_{v(i,j)})=1$.

\noindent (2) $p=\{(i,j),(i',j')\}, \ \s(a_{v(i,j)})=1, \ \s(a_{v(i',j')})=-1$. Also,   $i'\leq i$ and $j=j'$, or $i'=i$ and $j\leq j'$.

\noindent (3) $p=\{(i,j),(i-l,j),(i,j+l'),(i-l,j+l')\}, \ \s(a_{v(i,j)})=\s(a_{v(i-l,j+l')}) =1 , \ \s(a_{v(i-l,j)})=\s(a_{v(i,j+l')})=-1$, and $l,l'\geq 1$.
\end{definition}
 
Suppose $p$ is as in case (3). Then
\begin{align*}
\prod_{(i,j) \in p} a_{v(i,j)} & =   (1-\prod_{k=i}^j x_k)(1-\prod_{k=i-l}^j x_k)(1-\prod_{k=i}^{j+l'} x_k)(1-\prod_{k=i-l}^{j+l'} x_k) \\
& =   (1-\prod_{k=i}^j |x_k|)(1+\prod_{k=i-l}^j |x_k|)(1+\prod_{k=i}^{j+l'} |x_k|)(1-\prod_{k=i-l}^{j+l'} |x_k|) \\
& \leq   (1+\prod_{k=i}^j |x_k|)(1-\prod_{k=i-l}^j |x_k|)(1-\prod_{k=i}^{j+l'} |x_k|)(1+\prod_{k=i-l}^{j+l'} |x_k|)\ \ \\\
& =   \prod_{(i,j) \in p} a_{-|v|(i,j)},
\end{align*}
where the inequality follows from  lemma  \ref{NewPohst} (3) on setting $ 
a:=\prod_{k=i}^j x_k,$  $b:=\prod_{k=i-l}^{i-1} x_k$ and $c:=\prod_{k=j+1}^{j+l'}  x_k.$
The other cases are analogous, as we now record.
\begin{remark}\label{partition1} For $p$ as in definition \ref{defpartition}, we have  $\prod_{(i,j) \in p} a_{v(i,j)} \leq \prod_{(i,j) \in p} a_{-|v|(i,j)}$. 
\end{remark}

The following lemma will be central in our proof. However we will postpone its proof to the next section.

\begin{lemma}\label{partition}
For every $v=(x_1,\dots,x_n) \in \big([-1,1]\backslash\{0\}\big)^n$ the set $J_v$ has a good partition.
\end{lemma}

We can conclude the proof of the theorem by induction on $n$. If $n=1$, we obviously have $f_1(v)\leq 2$. If $n \geq 2$, let $v:=(x_1,\dots,x_n) \in [-1,1]^{n}$. If some of the coordinates $x_i$ is 0, then we note that for $v_1:=(x_1,\dots,x_{i-1})$ and $v_2:=(x_{i+1},\dots,x_{n})$ by induction we obtain
\begin{align*}
f_n(v) =  f_{i-1}(v_1) f_{n-i}(v_2)  \leq  2^{\lfloor \frac{(i-1)+1}{2}\rfloor} 2^{\lfloor \frac{(n-i)+1}{2}\rfloor} \leq  2^{\lfloor \frac{n+1}{2}\rfloor}.
\end{align*}
If $v=(x_1,\dots,x_n) \in \big([-1,1]\backslash\{0\}\big)^n$, we separate the product into the canonical elements and the non canonical elements, and we factor the non canonical elements using a partition obtained from lemma \ref{partition}. Thus, 
\begin{align*}
    f_n(x_1,\dots,x_n) & :=    \prod_{1\leq i \leq j \leq n} a_{v(i,j)} 
     =  \Big( \prod_{(i,j) \not \in J_v} a_{v(i,j)} \Big)\cdot  \Big( \prod_{(i,j)\in J_v} a_{v(i,j)} \Big) \\
    & =   \Big( \prod_{(i,j) \not \in J_v} a_{-|v|(i,j)} \Big) \cdot \Big( \prod_{p \in \pi_v} \prod_{(i,j) \in p} a_{v(i,j)} \Big) \qquad\quad(\text{see  remark }\ref{Notethis}) \\
    & \leq    \Big( \prod_{(i,j) \not \in J_v} a_{-|v|(i,j)} \Big) \cdot \Big( \prod_{p \in \pi_v} \prod_{(i,j) \in p} a_{-|v|(i,j)} \Big)\quad \quad(\text{see  remark }\ref{partition1} )\\
    & =   f_n(-|x_1|,\dots,-|x_n|)
\leq    2^{\lfloor \frac{n+1}{2}\rfloor} \qquad\qquad\quad\quad(\text{see  lemma }\ref{IdealCase}).
\end{align*}
As for the final claim in the theorem, we use lemma \ref{NewPohst} for each $p \in \pi_v$ and conclude with lemma \ref{IdealCase}.

\section{Proof of Lemma \ref{partition}}

The following remark, whose proof is a straightforward calculation, will be very useful.

\begin{remark}\label{signs} If in the set $\{(i,j),(i-l,j),(i,j+l'),(i-l,j+l')\}$, where $l,l'\geq 1$, three of the pairs are non canonical, then the last one is also non canonical. Moreover we have the relation $\s(a_{v(i,j)})\s(a_{v(i-l,j+l')})=s(a_{v(i-l,j)})\s(a_{v(i,j+l')})$.
\end{remark}

\begin{lemma}\label{Important}
If $v=(x_1,\dots,x_n) \in \big([-1,1]\backslash\{0\}\big)^n$, $n$ is even and if the number of $x_i$ with $\sign(x_i)=-1$ is odd, then $b_{1}(v)=b_{-1}(v)$, where $$ b_{l}(v):=|\{(i,j) \in J_v : \s(a_{v(i,j)})=l, i=1 \text{ or } j=n \}|.$$
\end{lemma}

\begin{proof}
Note that we just need to inspect the vertical and horizontal edges of the triangle in   Figure \ref{figure3}. We proceed by induction on $n$, the case $n=2$ being trivial. If  the sequence of signs of the $x_k$ is $(-,+,-,+,\ldots,-,+)$ or $ (+,-,+,-,\ldots,+,-)$, in which case necessarily $n\equiv2$ (mod 4), the lemma is a straight-forward calculation that does not even require the inductive assumption. If the $x_k$ do not have this sign pattern, then the sequence of signs   must necessarily contain  consecutive $++$ or $--$.  The inductive step amounts to the following 
\vskip0.3cm
\noindent {\it{Claim}}. If $b_{1}(v)=b_{-1}(v)$ for some $v \in \big([-1,1]\backslash\{0\}\big)^n$, then $b_{1}(v')=b_{-1}(v')$ for $v':=(x_1,\dots,x_k,e_1$ $,e_2,x_{k+1},\dots,x_n) \in \big([-1,1]\backslash\{0\}\big)^{n+2}$ where $\sign(e_1)=\sign(e_2)$ and $0 \leq k \leq n$.
\begin{figure}[h!] 
    \centering
    \begin{tikzpicture}
        \matrix [matrix of math nodes] (m)
        {
            a_{v'(1,n+2)} &a_{v'(2,n+2)} &\hdots &a_{v'(k,n+2)} &a_{v'(k+1,n+2)} &a_{v'(k+2,n+2)} &\hdots  &a_{v'(n+2,n+2)}  \\
            \vdots& \vdots & &\vdots  &\vdots &\vdots & \iddots & \\
            a_{v'(1,k+2)} &a_{v'(2,k+2)} &\hdots &a_{v'(k,k+2)} &a_{v'(k+1,k+2)} & e_2 & &  \\    
            a_{v'(1,k+1)} &a_{v'(2,k+1)} &\hdots &a_{v'(k,k+1)} & e_1 & & & \\      
            a_{v'(1,k)} &a_{v'(2,k)} &\hdots &a_{v'(k,k)} & & & & \\     
            \vdots &\vdots &\iddots & & & & & \\    
            a_{v'(1,2)} &a_{v'(2,2)} & & & & & & \\     
            a_{v'(1,1)} & & & & & & & \\      
        };  
        \draw[color=red] (m-4-1.south west) rectangle (m-3-1.north east);
        \draw[color=red] (m-1-5.south west) rectangle (m-1-6.north east);
\end{tikzpicture}
    \caption{Elements for which we must inspect the product sign.}
    \label{figure3}
\end{figure}

To prove this, note that if $j\leq k$ then $a_{v'(i,j)}=a_{v(i,j)}$. Also, if $i \geq k+3$ then $a_{v'(i,j)}=a_{v(i-2,j-2)}$. Thus in both cases $\s(a_{v'(i,j)})=\s(a_{v(i,j)})$ and $\s(a_{v'(i,j)})=\s(a_{v(i-2,j-2)})$. Moreover, if $i \leq k$ and $j \geq k+3$, then $\s(a_{v'(i,j)})= \s(a_{v(i,j-2)})$. So we just need to inspect the elements $a_{v'(1,k+1)}$, $a_{v'(1,k+2)}$, $a_{v'(k+1,n+2)}$ and $a_{v'(k+2,n+2)}$.

If $\sign(e_1)=1$, then $\s(a_{v'(1,k+1)})= \s(a_{v'(1,k+2)})$ and therefore $a_{v'(1,k+1)} \in J_v$ if and only if $a_{v'(1,k+2)} \not \in J_v$. The same statement holds for $a_{v'(k+1,n+2)}$ and $a_{v'(k+2,n+2)}$. Also, since the number of $\{x_i\}_{i=1}^n$ with $\sign(x_i)=-1$ is odd, then $\sign(\prod_{i=1}^k x_i)=-\sign(\prod_{i=k+1}^{n} x_i)$. This implies that $\s(a_{v'(1,k+1)}) \neq s(a_{v'(k+1,n+2)})$, so we have $b_{1}(v')=b_{-1}(v')$. 

If $\sign(e_1)=-1$, then $\s(a_{v'(1,k+1)}) \neq  \s(a_{v'(1,k+2)})$ and therefore $a_{v'(1,k+1)} \in J_v$ if and only if $a_{v'(1,k+2)} \in J_v$. The same statement holds for $a_{v'(k+1,n+2)}$ and $a_{v'(k+2,n+2)}$. So we have $b_{1}(v')=b_{-1}(v')$.
\end{proof}

We will construct the partition of $J_v$ inductively. Let us define a total order over $\{(i,j): i\leq j\}$ as in Figure \ref{figure4}, $(i,j) \prec (i',j')$ if and only if $j' > j$, or $j'=j$ and $i' < i$. 
\begin{figure}[h] 
    \centering
    \begin{tikzpicture}
    \matrix [matrix of math nodes] (m)
        {
            (1,n) &(2,n) &(3,n) &(4,n) &(5,n) &(6,n) &\hdots  &(n,n)  \\
            \vdots& \vdots &\vdots &\vdots &\vdots &\vdots & \iddots & \\
            (1,6) &(2,6) &(3,6) &(4,6) &(5,6) &(6,6) & &  \\    
            (1,5) &(2,5) &(3,5) &(4,5) &(5,5) & & & \\      
            (1,4) &(2,4) &(3,4) &(4,4) & & & & \\     
            (1,3) &(2,3) &(3,3) & & & & & \\    
            (1,2) &(2,2) & & & & & & \\     
            (1,1)\\      
        };  
       \draw[color=red][->] (m-7-2) edge (m-7-1);
       \draw[color=red][->]  (m-6-3) edge (m-6-2);
       \draw[color=red][->]  (m-6-2) edge (m-6-1);
       \draw[color=red][->]  (m-5-4) edge (m-5-3);
       \draw[color=red][->]  (m-5-3) edge (m-5-2);
       \draw[color=red][->]  (m-5-2) edge (m-5-1);
       \draw[color=red][->]  (m-4-5) edge (m-4-4);
       \draw[color=red][->]  (m-4-4) edge (m-4-3);
       \draw[color=red][->]  (m-4-3) edge (m-4-2);
       \draw[color=red][->]  (m-4-2) edge (m-4-1);
       \draw[color=red][->]  (m-3-6) edge (m-3-5);
       \draw[color=red][->]  (m-3-5) edge (m-3-4);
       \draw[color=red][->]  (m-3-4) edge (m-3-3);
       \draw[color=red][->]  (m-3-3) edge (m-3-2);
       \draw[color=red][->]  (m-3-2) edge (m-3-1);
       \draw[color=red][->]  (m-1-7) edge (m-1-8);
       \draw[color=red][->]  (m-1-6) edge (m-1-7);
       \draw[color=red][->]  (m-1-5) edge (m-1-6);
       \draw[color=red][->]  (m-1-4) edge (m-1-5);
       \draw[color=red][->]  (m-1-3) edge (m-1-4);
       \draw[color=red][->]  (m-1-2) edge (m-1-3);
       \draw[color=red][->]  (m-1-1) edge (m-1-2);
       \draw[color=red][->]  (m-8-1) edge[out=170,in=350] node[yshift=0.3ex] { } (m-7-2);
       \draw[color=red][->]  (m-7-1) edge[out=170,in=350] node[yshift=0.3ex] { } (m-6-3);
       \draw[color=red][->]  (m-6-1) edge[out=170,in=350] node[yshift=0.3ex] { } (m-5-4);
       \draw[color=red][->]  (m-5-1) edge[out=170,in=350] node[yshift=0.3ex] { } (m-4-5);
       \draw[color=red][->]  (m-4-1) edge[out=170,in=350] node[yshift=0.3ex] { } (m-3-6);
       \draw[color=red][->]  (m-3-1) edge[out=170,in=350,dotted] node[yshift=0.3ex] { } (m-2-7);
       \draw[color=red][->]  (m-2-1) edge[out=170,in=350,dotted] node[yshift=0.3ex] { } (m-1-8);
\end{tikzpicture}
    \caption{Order used to construct the partition of $J_v$.}
    \label{figure4}
\end{figure}
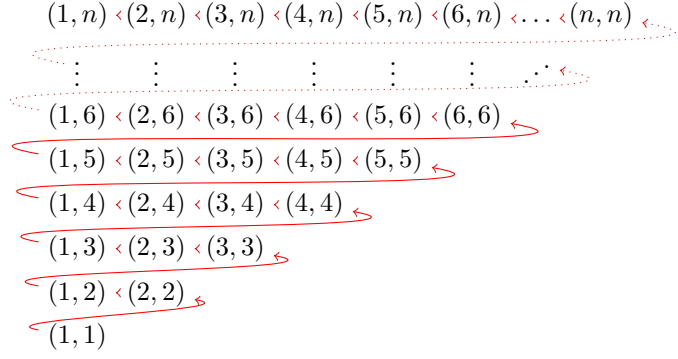

\noindent
We define $\pi_0 \subset \mathcal{P}(J_v)$, a subset of the power set of $J_v$, as $$\pi_0:=\big\{ \{(i,j)\} \big| \, \s(a_{v(i,j)})=1=(-1)^{i+j},\  1\leq i\leq j\leq n \big\}.$$

\noindent
Let $N\ge0$ be the number of pairs $(i,j)\in J_v$ with $\s(a_{v(i,j)})=-1$. If $N=0$, then $\pi_0$ is already a good partition. If $N>0$, for $1\leq k \leq N$ we will add inductively to $\pi_{k-1}$ the $k$-th $(i,j)$ such that $a_{v(i,j)}$ has negative product sign.  Thus (some element of) $\pi_k$  contains the first $k$ pairs $(i,j)$ in the order $\prec$. We will show that $\pi_N$ is the partition of $J_v$ claimed in lemma \ref{partition}. 

If $(i,j)$ is the $k$-th element of $J_v$ for which $s(a_{v(i,j)})=-1$ we will choose one of two operations to apply to $\pi_{k-1}$ to produce $\pi_k$. 
\begin{enumerate}
    \item[ ] Operation 1. If $\{(i',j')\} \in \pi_{k-1}$ with $i'=i$ and $i \leq j' < j$, or $j'=j$ and $i < i' \leq j$, where $a_{v(i',j')}$ has positive product sign, then
$$
\pi_k:=\Big(\pi_{k-1}-\big\{\{(i',j')\}\big\}\Big)\cup\Big\{\{(i',j'),(i,j)\}\Big\}.
$$
    \item[ ] Operation 2. If  $\{(i,l),(r,l)\}\in \pi_{k-1}$ and $\{(r,j)\} \in \pi_{k-1}$ with $i\leq l < j$ and $1 \leq r < i$, where $a_{v(i,l)}$ and $a_{v(r,j)}$ have positive product sign and $a_{v(r,l)}$ has negative product sign, then 
$$
\pi_k:=\Big(\pi_{k-1}-\big\{\{(i,l),(r,l)\},\{(r,j)\}\big\}\Big)\cup\Big\{\{(i,l),(r,l),(i,j),(r,j)\}\Big\}.
$$
\end{enumerate}
Thus operation 1 removes a singleton from $\pi_{k-1}$ and inserts a doubleton containing this singleton. Operation 2 removes a doubleton and a singleton  from $\pi_{k-1}$ and inserts a quadrupleton containing the removed elements and forming the vertices of a rectangle. 
It is useful to visualize the effect of the operations as in Figure \ref{figure5}, where the  subindex gives the sign of the corresponding product.

\begin{figure}[h]
    \centering
     \begin{tikzpicture}
    \matrix [matrix of math nodes] (m)
        {
            (i,j)_{-} &\hdots &(i',j')_{+}  \\
            \vdots& & \\
            (i',j')_{+}& & \\
        };  
\end{tikzpicture}    \qquad \qquad \qquad
     \begin{tikzpicture}
    \matrix [matrix of math nodes] (m)
        {
            (r,j)_{+} &\hdots &(i,j)_{-}  \\
            \vdots&  &\vdots \\
            (r,l)_{-}& \hdots &(i,l)_{+} \\
        };  
\end{tikzpicture}    \caption{Operations 1 and 2.}
    \label{figure5}
\end{figure}

\begin{remark}\label{FollowTheOrder}
After either operation the new set $\pi_k$ only contains sets $p$ that correspond to one of the cases in definition \ref{defpartition} and $\sum_{p \in \pi_k} |p|=1+\sum_{p \in \pi_{k-1}} |p|$. Furthermore,  all $(l,t)\in J_v $ with $(l,t)\prec(i,j)$ are already in $\pi_{k-1}$ (regardless of the product sign of $a_{v(l,t)}$), but $(l,t)$ is not contained in $\pi_{k-1}$ if $(i,j) \preceq (l,t)$ and $s(a_{v(l,t)})=-1$.
\end{remark}

For simplicity we are going to say that a non-canonical pair $(i,j)$ is positive (respectively negative) if the corresponding term $a_{v(i,j)}$ has positive (respectively negative) product sign. If we already know the product sign of a pair we will add it as a sub-index. The following definition will also be useful.

\begin{definition}\label{cases}
Let $0 \leq k \leq N$. If $(i,j)_+ \in J_v$ is positive, then we will say that
\begin{enumerate}
    \item $(i,j)_+$ is in a singleton if $\{(i,j)_+\} \in \pi_k$. We call it a $\sing$-configuration.

    \item $(i,j)_+$ is in an h-doubleton if $\{(i,j)_+,(i-l,j)_-\} \in \pi_k$, $i-l < i$. We call it a  $\hdoub$-configuration.

    \item $(i,j)_+$ is in a v-doubleton if $\{(i,j)_+,(i,j+l)_-\} \in \pi_k$, $j<j+l$. We call it a  $\vdoub$-configuration.

    \item $(i,j)_+$ is in an i-quadrupleton if $\{(i,j)_+,(i-l_1,j)_-,(i,j+l_2)_-,(i-l_1,j+l_2)_+\} \in \pi_k$, $l_1,l_2 \geq 1$. We call it a  $\iquad$-configuration.

    \item $(i,j)_+$ is in a t-quadrupleton if $\{(i,j)_+,(i+l_1,j)_-,(i,j-l_2)_-,(i+l_1,j-l_2)_+\} \in \pi_k$, $l_1,l_2 \geq 1$ and $i+l_1\leq j-l_2$.  We call it a $\tquad$-configuration.

\end{enumerate}

If $(i,j)_-$ is negative and is contained in some subset of $\pi_k$, then we will say that

\begin{enumerate}
    \item $(i,j)_-$ is in an h-doubleton if $\{(i,j)_-,(i+l,j)_+\} \subset p \in \pi_k$, $i<i+l\leq j$.  We call it a $\nhdoub$-configuration.
    \item $(i,j)_-$ is in a v-doubleton if $\{(i,j)_-,(i,j-l)_+\} \subset p \in \pi_k$, $i\leq j-l< j$. We call it a $\nvdoub$-configuration.
\end{enumerate}

Finally, given $(i,j)_-$ and $(i,l)_+$, $i\leq l <j$, in $\hdoub$-configuration with $(i',l)_-$, $1 \leq i'\leq l$. We say that we can apply operation 2 when the pair $(i',j)_+$ is in $\sing$-configuration.
\end{definition}

Now we show the method used to obtain $\pi_{k}$ from $\pi_{k-1}$ for each $k$, $1 \leq k \leq N$. Suppose $(i,j)_-$ is the $k$-th negative pair and fix the horizontal list  $$l_{(i,j)}:=\{(i,j),(i+1,j),\dots,(j,j)\}\cap J_v,$$
we do one of the following
\begin{itemize}
    \item \textbf{Case 1.} If there is some positive pair in the list $l_{(i,j)}$ in a $\sing$-configuration then we will use operation 1 with the maximal (with the order $\prec$) positive pair $(i+l,j)_+$, $i<i+l\leq j$, contained in the list $l_{(i,j)}$ that is in a $\sing$-configuration. 
    \item \textbf{Case 2.} If $(i,j)_-$ is the minimal negative pair on the list $l_{(i,j)}$ for the which we couldn't apply Case 1. We will prove (Lemma \ref{step2}) that there is a unique positive pair in the vertical list $s_{(i,j)}:=\{(i,i),\dots,(i,j-1),(i,j)\}\cap J_v$ that is contained in a $\sing$-configuration or in a $\hdoub$-configuration such that we can apply operation 2. In the first instance we use operation $1$ while in the second instance we use operation 2. 
    \item \textbf{Case 3.} If we couldn't apply Case 1 nor Case 2, we consider the minimal negative pair $(i_1,j)_-$ on the list $l_{(i,j)}$ for the which we couldn't apply Case 1. From Case 2, $(i_1,j)_-$ is contained in a $\nvdoub$-configuration with some positive pair $(i_1,j_1)_+$. We will prove (Lemma \ref{step3}) that the positive pair $(i,j_1)_+$ is contained in a $\sing$-configuration or in a $\hdoub$-configuration such that we can apply operation 2. In the first instance we use operation $1$ while in the second instance we use operation 2. 
\end{itemize}

Proving lemmas \ref{step2} and \ref{step3} will conclude our proof of lemma \ref{partition}. 

\begin{remark}\label{notpossible}
Let $0 \leq k \leq N$, suppose that we are able to construct $\pi_k$ following the method above. The following five configurations are impossible in $\pi_k$.
\begin{enumerate}
    \item $(i,j)_-$ is in a $\nhdoub$-configuration with $(i+l,j)_+$, $(i',j)_-$ is in a $\nhdoub$-configuration with $(i'+l',j)_+$ and $i < i' <i+l < i'+l'$. 
    \item $(i,j)_-$ is in a $\nhdoub$-configuration with $(i+l,j)_+$, $(i',j)_-$ is in a $\nvdoub$-configuration with $(i',j-l')_+$ and $i < i' <i+l$.
    \item  $(i,j)_-$ is not in a $\nhdoub$-configuration, $(i',j)_+$ is in a $\vdoub$-configuration with $(i',j+l)_-$ and $i<i'$.
    \item $(i,j)_-$ is in a $\nvdoub$-configuration with $(i,j-l_1)_+$, $(i',j)_-$ is in a $\nvdoub$-configuration with $(i',j-l_2)_+$ $i<i'$ and $l_1\neq l_2$.
    \item $(i,j)_-$ is in a $\nvdoub$-configuration with $(i,j-l_1)_+$, $(i',j)_-$ is in a $\nvdoub$-configuration with $(i',j-l_2)_+$,  $i<i'$ and $(i,j)_-$ is the minimal negative pair on the list $l_{(i,j)}$ for the which we couldn't apply Case 1.
\end{enumerate}
\end{remark}
The first three items are consequences of Case 1 of the construction given above, item \textit{(4)} is a consequence of Case 3 and item  \textit{(5)} is a consequence of Case 2.  See Figure \ref{badconf}.

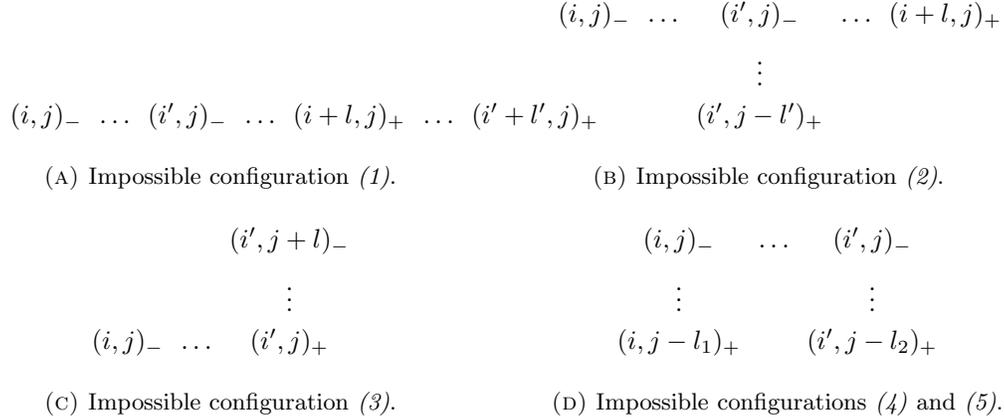
\begin{figure}[h]
    \centering
    \begin{subfigure}[t]{0.45\textwidth}
    \begin{tikzpicture}
    \matrix [matrix of math nodes] (m)
        {
            (i,j)_{-} &\hdots &(i',j)_{-} & \hdots & (i+l,j)_{+} & \hdots & (i'+l',j)_{+}  \\
            & & & & & &  \\
            & & & & & &\\
        };  
\end{tikzpicture} \caption{Impossible configuration \textit{(1)}.}   
    \label{badconf1}
    \end{subfigure}
    \hfill \hfill
    \begin{subfigure}[t]{0.45\textwidth}
        \centering
    \begin{tikzpicture}
    \matrix [matrix of math nodes] (m)
        {
            (i,j)_{-} & \hdots &(i',j)_{-} & \hdots & (i+l,j)_{+}  \\
            &  & \vdots & &  \\
            &  & (i',j-l')_{+} & & \\
        };  
\end{tikzpicture}    \caption{Impossible configuration \textit{(2)}.}
    \label{badconf2}
    \end{subfigure}
    \vspace{0.2cm}

    \centering
    \begin{subfigure}[t]{0.45\textwidth}
        \centering
     \begin{tikzpicture}
    \matrix [matrix of math nodes] (m)
        {
            &  &(i',j+l)_{-}\\
            &  & \vdots  \\
            (i,j)_-& \hdots  & (i',j)_{+}\\
        };  
\end{tikzpicture}    \caption{Impossible configuration \textit{(3)}.}
    \label{badconf3}
    \end{subfigure}
    \hfill
   \begin{subfigure}[t]{0.45\textwidth}
    \centering
    \begin{tikzpicture}
    \matrix [matrix of math nodes] (m)
        {
           (i,j)_- & \hdots & (i',j)_-\\
           \vdots & & \vdots\\
           (i,j-l_1)_+ &  & (i',j-l_2)_+\\
        };  
\end{tikzpicture} \caption{Impossible configurations \textit{(4)} and \textit{(5)}.}   
    \label{badconf4}
    \end{subfigure}
    
    \caption{Configurations from remark \ref{notpossible}.}
        \label{badconf}
\end{figure}

In fact our construction makes a variety of configurations impossible.

\begin{lemma}\label{morethan}
Let $0 \leq k \leq N$,
suppose that we are able to construct $\pi_k$ following the method above and consider a horizontal list  $L:=\{(i,j),\dots,(i+l',j)\} \cap J_v$.
\begin{enumerate}
    \item If $L$ contains strictly more positive pairs than negative pairs, then there exist some positive pair $(i+l,j)_+ \in L$ contained in a partition $p$ such that $p \cap L= \{(i+l,j)_+\}$.
    \item If $L$ contains strictly more negative pairs than positive pairs and every negative pair in $L$ is contained in a partition, then there exist some negative pair $(i+l,j)_- \in L$ contained in a partition $p$ such that $p \cap L= \{(i+l,j)_-\}$, in particular $(i+l,j)_-$ must be in a $\nvdoub$-configuration.
\end{enumerate}
We have analogues statements for vertical lists.
\end{lemma}
\begin{proof}
Suppose that for every positive pair $(i+l,j)_+ \in L$ there exist a partition $p$ such that $p \cap L \neq \{(i+l,j)_+\}$. By inspecting the different cases of definition \ref{cases} we conclude that there exist some $b_l$ such that $p \cap L = \{(i+l,j)_+, (i+b_l,j)_-\}$. Since the partitions are disjoint we are able to find a negative pair for each positive pair, which contradicts the hypothesis. The proof of the second statement is analogous. 
\end{proof}

\begin{lemma}\label{configuration}
Let $0 \leq k \leq N$,
suppose that we are able to construct $\pi_k$ following the method above. If there is some non canonical negative pair $(i,j)_-$ not in a $\nhdoub$-configuration, every pair $(a,b) \prec (i,j)_-$ is contained in a partition of $\pi_k$ and there is a negative pair $(i,j-l_2)$ in a $\nhdoub$-configuration with the positive pair $(i+l_1,j-l_2)_+$, then in the list $L:=\{(i,j)_-,(i+1,j),\dots,(i+l_1,j)_+\}\cap J_v$ there is at least one positive pair $(i+l,j)_+$ in a $\sing$-configuration or in a $\tquad$-configuration. 
Moreover, if $(i+l,j)_+$ is in a $\tquad$-configuration with $(i+l+l_1',j)_-$, $(i+l,j-l_2')_-$ and $(i+l+l_1',j-l_2')_+$, then $i+l_1<i+l+l_1'$ (See figure \ref{figure13}).
\end{lemma}
\begin{proof}
We will prove that in the list $L_{sub}:=\{(i,j-l_2)_-,(i+1,j-l_2),\dots,(i+l_1,j-l_2)_+\}\cap J_v$ there are more positive pairs than negative pairs. Suppose that there is some negative pair $(i+l',j-l_2)_-$ in $L_{sub}$, $0\leq l'<l_1$,  then by remark \ref{notpossible} \textit{(1)} and \textit{(2)} it must be in a $\nhdoub$-configuration with some $(i+l'',j-l_2)_+$ such that $i+l''\leq i+l_1$. So for each negative pair in $L_{sub}$ there is a unique corresponding positive pair in $L_{sub}$.

Remark \ref{signs} implies that $s(a_{v(i+r,j)})=s(a_{v(i+r,j-l_2)})$ for $0 \leq r\leq l_1$, thus there are strictly more positive elements than negative elements in the list $L':=L-\{(i,j_-)\}$. By lemma \ref{morethan} there must be at least one positive pair $(i+l,j)_+ \in L$ contained in a partition $p \in \pi_k$ without a negative pair in the list $L'$, \ie $p\cap L'=\{(i+l,j)_+\}$.  We conclude inspecting the possible configurations for $(i+l,j)_+$ of definition \ref{cases}, $\hdoub$-configuration, $\vdoub$-configuration and $\iquad$-configuration form impossible configurations \textit{(2)} and \textit{(3)} of remark \ref{notpossible} and if $(i+l,j)_+$ is in a $\tquad$-configuration with $(i+l+l_1',j)_-$, $(i+l,j-l_2')_-$ and $(i+l+l_1',j-l_2')_+$, then $(i+l+l_1',j)_- \not \in L'$, which is equivalent to the last statement of the lemma. 
\end{proof}

\begin{figure}[h!] 
    \centering
    \begin{tikzpicture}
    \matrix [matrix of math nodes] (m)
        {
(i,j)_{-} & \hdots & (i+l,j) _{+}& \hdots & (i+l_1,j)_{+} & \hdots & (i+l+l_1',j)_{-}\\
 \vdots & &  & &\vdots & & \\
(i,j-l_2)_{-}& & \hdots & &(i+l_1,j-l_2)_{+} & & \\
        };  
\end{tikzpicture}
    \caption{Lemma \ref{configuration} when \ensuremath{(i+l,j)_+} is in a \protect\tquad-configuration.}
    \label{figure13}
\end{figure}

In the rest of the paper we let $1 \leq k \leq N$, suppose that we are able to construct $\pi_{k-1}$ following the method above and we fix $(i,j)_-$ the $k$-th negative pair.

\begin{lemma}\label{step2}
Suppose $(i,j)_-$ is the minimal negative pair on the list $l_{(i,j)}$ for the which we couldn't apply Case 1, then in the vertical list $s_{(i,j)}$ there is a unique positive pair $(i,j')_+$ for the which we can apply operation 1 or operation 2. Specifically we will prove that $(i,j')_{+}$ is in a $\sing$-configuration (So we can apply operation 1) or in a $\hdoub$-configuration with $(i',j')_-$ and $(i',j)_+$ is in a $\sing$-configuration (So we can apply operation 2). 
\end{lemma}
\begin{proof}
\textbf{Existence.}
\noindent
We will use lemma \ref{Important}.  Since $s(a_{v(i,j)})=-1$, there is an odd number of negative $x_r$'s with $i\le r\le j$. As we are also assuming that $a_{v(i,j)}$ is non canonical,  $j-i $ is odd. Thus 
the hypotheses of lemma \ref{Important} with the vector $v'=(x_i,x_{i+1},\dots,x_{j})$ are satisfied.\footnote{\ To  apply  lemma \ref{Important} to $v'$, note that    $a_{v(l,t)}=a_{v'(l-i+1,t-i+1)}\ (i\le l \le t \le j)$, so that $(l,t)$ is canonical with respect to $v$ if and only if $(l-i+1,t-i+1)$ is canonical with respect to $v'$.}  Hence,  one of the following holds.

\vskip.2cm
\noindent $\bullet$ Subcase 1. List $A_1:=s_{(i,j)_-}-\{(i,j)\}=\{(i,i),(i,i+1),\dots,(i,j-1) \}\cap J_v$ has strictly more positive pairs than negative pairs. 

\vskip.1cm

\noindent $\bullet$ Subcase 2. List $A_2:=l_{(i,j)_-}-\{(i,j)\}=\{(i+1,j),\dots,(j-1,j),(j,j) \}\cap J_v$ has strictly more positive pairs than negative pairs. 
\vskip.2cm

Subcase 2 is impossible, if $A_2$ has strictly more positive pairs than negative pairs then by lemma \ref{morethan} there must be some positive pair in the list $l_{(i,j)}$ in $\sing$-configuration or in $\vdoub$-configuration. The first possibility is a contradiction with the fact that we couldn't apply Case 1 of the construction and the second possibility is impossible since we would obtain impossible configuration \textit{(4)} of remark \ref{notpossible}. 

So $A_1$ has strictly more positive pairs than negative pairs, by lemma \ref{morethan} it must have a positive pair $(i,j')_+$ in $\sing$-configuration or in $\hdoub$-configuration with $(i',j')_-$, for $i'<i$. If $(i,j')_+$ is in  $\sing$-configuration we can apply operation 1. If $(i,j')_+$ is in $\hdoub$-configuration we have to prove that $(i',j)_+$ is in $\sing$-configuration so that we can apply operation 2. Inspecting the different configurations for the positive pair $(i',j)_+$, by remark \ref{FollowTheOrder}, it can't be in a $\vdoub$-configuration, $\hdoub$-configuration or $\iquad$-configuration. Suppose it is in a $\tquad$-configuration (See figure \ref{existencefigure}) with some pairs $(i'',j)_{-}$, $(i'',j'')_{+}$ and $(i',j'')_{-}$, then by remark \ref{FollowTheOrder} $i<i''$. Thus $(i'',j)_{-}$ is in a $\nvdoub$-configuration with $(i'',j'')_{+}$. Since $i<i''$ we would obtain impossible configuration \textit{(5)} of remark \ref{notpossible}. We conclude that $(i',j)_+$ must be in $\sing$-configuration and we can apply operation 2.

\begin{figure}[h!]
\centering
\begin{minipage}{.5\textwidth}
  \centering
    \centering
    \begin{tikzpicture}
    \matrix [matrix of math nodes] (m)
        {
            (i',j)_+&\hdots &(i,j)_- &\hdots &(i'',j)_-\\
            \vdots& &\vdots & &\\
            (i',j')_-&\hdots &(i,j')_+ & &\\
        };  
\end{tikzpicture}
    \caption{If \ensuremath{(i',j)_+} is in \protect\tquad-configuration.}
    \label{existencefigure}
\end{minipage}%
\begin{minipage}{.5\textwidth}
  \centering
    \begin{tikzpicture}
    \matrix [matrix of math nodes] (m)
        {
            (i,j)_- &\hdots &(i+l,j)_+\\
            \vdots & &\vdots\\
            (i,j')_- &\hdots &(i+l,j')_+\\
        };  
\end{tikzpicture}
    \caption{If $(i,j')_+$ is not unique.}
    \label{uniquenessfigure}
\end{minipage}
\end{figure}

\textbf{Uniqueness.}
\noindent
We already proved that in $A_2$ there are more negative pairs than positive pairs. However, if there are strictly more negative pairs than positive pairs, then by lemma \ref{morethan} we can find some negative pair in $A_2$ in a $\nvdoub$-configuration and would obtain impossible configuration \textit{(5)} of remark \ref{notpossible}. So in the list $A_2$ there is the same number of positive and negative pairs. Now lemma $\ref{Important}$ guarantees that $s_{(i,j)}$ has exactly the same number of positive and negative pairs.

Suppose there are two positive pairs in $\sing$-configuration or in $\hdoub$-configuration in the list $A_1$. If every negative pair in the list $A_1$ is in $\nvdoub$-configuration, since the partitions are disjoint, we would obtain strictly more positive pairs than negative pairs in $s_{(i,j)}$. So there is at least one negative pair $(i,j')_{-}$ in $\nhdoub$-configuration with $(i+l,j')_{+}$ in the list $A_1$ (See figure \ref{uniquenessfigure}). Applying lemma \ref{configuration} to $(i,j)_-$ and $(i,j')_-$ in $\nhdoub$-configuration with $(i+l,j')_+$ we obtain a positive pair in the list $l_{(i,j)}$ in a $\sing$-configuration or in a $\tquad$-configuration. The first possibility is a contradiction with the fact the we couldn't apply Case 1 of the construction and the second possibility is impossible since we would obtain impossible configuration \textit{(5)} of remark \ref{notpossible}.\end{proof}

\begin{lemma}\label{step3}
Suppose $(i,j)_-$ is not the minimal negative pair on the list $l_{(i,j)}$ for the which we couldn't apply Case 1, that $(i_1,j)_-$ is the minimal negative pair on the list $l_{(i,j)}$ for the which we couldn't apply Case 1 and that $(i_1,j)_-$ is contained in a $\nvdoub$-configuration with some positive pair $(i_1,j_1)_{+}$, then it is possible to apply operation 1 or operation 2 with the positive pair $(i,j_1)_{+}$. Specifically we will prove that $(i,j_1)_{+}$ is in a $\sing$-configuration (So we can apply operation 1) or in a $\hdoub$-configuration with $(i-l,j_1)_-$ and $(i-l,j)_+$ is in a $\sing$-configuration (So we can apply operation 2). 
\end{lemma}
\begin{proof}
We will study separately each possible configuration for the pair $(i,j_1)_{+}$.

\begin{figure}[h]
    \begin{subfigure}[t]{\textwidth}
    \centering
    \begin{tikzpicture}
    \matrix [matrix of math nodes] (m)
        {
            (i-l,j)_{+} &\hdots &(i,j)_{-} &\hdots &(i',j)_{-}  \\
            \vdots & & \vdots & & \vdots \\
            (i-l,j_1)_{-} &\hdots &(i,j_1)_{+} &\hdots &(i',j')_{+} \\    
        };  
\end{tikzpicture}
    \caption{If $(i,j_1)_{+}$ is in a $\hdoub$-configuration.}
    \label{AlmostLast}
    \end{subfigure}
    \vspace{0.2cm}

    \centering
    \begin{subfigure}[t]{0.45\textwidth}
        \centering
    \begin{tikzpicture}
    \matrix [matrix of math nodes] (m)
        {
            (i,j)_{-} &\hdots &(i_1,j)_{-} &\hdots &(i',j)_{+}  \\
            \vdots & & \vdots & & \vdots \\
            (i,j')_{-} &\hdots &(i_1,j')_{-} &\hdots &(i',j')_{+} \\    
             \vdots& &\vdots & &  \\      
            (i,j_1)_{+}& \hdots &(i_1,j_1)_{+} & & \\ 
        };  
\end{tikzpicture}
    \caption{If $(i,j_1)_+$ is in a $\vdoub$-configuration or in a $\iquad$-configuration.}
    \label{figure7}
    \end{subfigure}
    \hfill
    \begin{subfigure}[t]{0.45\textwidth}
        \centering
    \begin{tikzpicture}
    \matrix [matrix of math nodes] (m)
        {
            (i,j)_{-} &\hdots &(i_1,j)_{-} &\hdots &(i',j)_{+}  \\
            \vdots & & \vdots & & \vdots \\
            (i,j_1)_{+} &\hdots &(i_1,j_1)_{+} &\hdots &(i',j_1)_{-} \\    
             \vdots& &\vdots & & \vdots \\      
            (i,j')_{-}& \hdots &(i_1,j')_{-} & \hdots &(i',j')_+ \\ 
        };  
\end{tikzpicture}
    \caption{If $(i,j_1)_{+}$ is in a $\tquad$-configuration.}
    \label{LastOne}
    \end{subfigure}
    \caption{Figures of lemma \ref{step3}.}
        \label{THELASTF}
\end{figure}
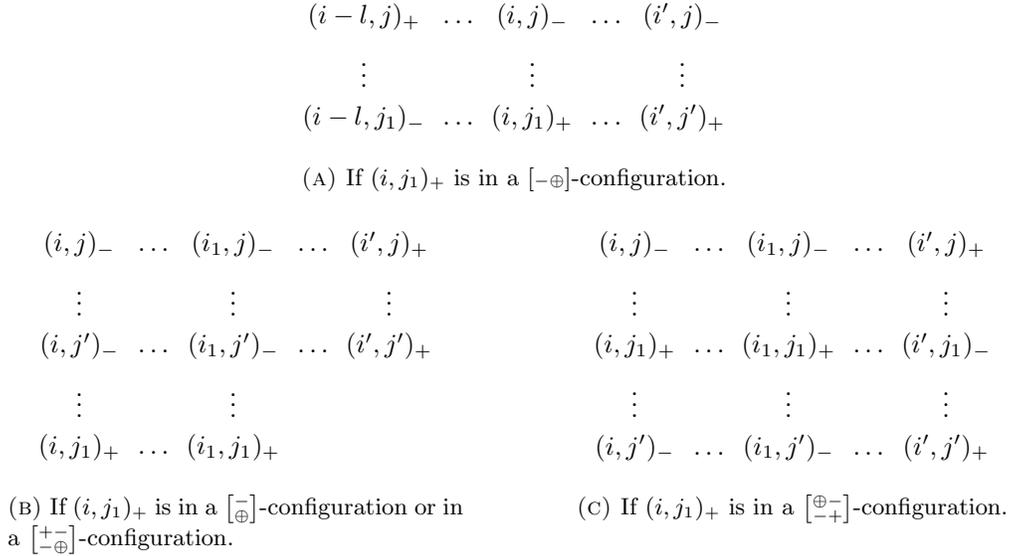

\begin{itemize}
    \item If $(i,j_1)_{+}$ is in a $\sing$-configuration we can directly apply operation 1.
    
    \item If $(i,j_1)_{+}$ is in a $\hdoub$-configuration with $(i-l,j_1)_-$ we have to prove that $(i-l,j)_+$ is a $\sing$-configuration, so we inspect the different configurations for the pair $(i-l,j)_+$ (See figure \ref{AlmostLast}). By remark \ref{FollowTheOrder} it can't be in a $\vdoub$-configuration, $\hdoub$-configuration or $\iquad$-configuration. Suppose it is in a $\tquad$-configuration with some pairs $(i',j)_-$, $(i-l,j')_-$ and $(i',j')_+$, then to avoid  impossible configuration \textit{(4)} of remark \ref{notpossible} necessarily $j'=j_1$. So the pair $(i-l,j_1)_-$ would be contained in two different partitions, which is impossible. 

    \item If $(i,j_1)_{+}$ is in a $\vdoub$-configuration or in a $\iquad$-configuration (See figure \ref{figure7}), then there is some $(i,j')_-$ in a $\nvdoub$-configuration with $(i,j_1)_{+}$ and remark \ref{signs} implies that the pair $(i_1,j')_-$ is negative.
    
    If $(i_1,j')_-$ was in $\nvdoub$-configuration with $(i_1,j'')_{+}$, since $(i,j')_-$ in a $\nvdoub$ configuration with $(i,j_1)_{+}$, to avoid impossible configuration \textit{(4)} of remark \ref{notpossible} necessarily $j''=j_1$ and so  the pair $(i_1,j_1)_-$ would be contained in two different partitions, which is impossible. Thus $(i_1,j')_-$ is in a $\nhdoub$-configuration with some pair $(i',j')_+$
    
    We finally apply lemma \ref{configuration} to the negative pair $(i_1,j)_-$ and to $(i_1,j')_-$ in a $\nhdoub$-configuration with $(i',j')_+$. We conclude that there must be some positive pair $(i''.j)_+$ with $i_1<i''$ in a $\sing$-configuration or in a $\tquad$-configuration, however both possibilities are impossible. The first possibility is a contradiction with the fact the we couldn't apply Case 1 of the construction and the second possibility is impossible since we would obtain impossible configuration \textit{(5)} of remark \ref{notpossible}.

    \item If $(i,j_1)_{+}$ is in a $\tquad$-configuration with $(i',j_1)_-$, $(i,j')_-$ and $(i',j')_+$ (See figure \ref{LastOne}). Remark \ref{notpossible} \textit{(2)} implies $i_1<i'$. By remark $\ref{signs}$ the pair $(i_1,j')_-$ is negative and by remark \ref{notpossible} \textit{(2)} the pair $(i_1,j')_-$ must be in $\nhdoub$-configuration with some positive pair $(i'',j')_+$. We apply lemma \ref{configuration} to $(i_1,j)_-$ and the pair $(i_1,j')_-$ in $\nhdoub$-configuration with $(i'',j')_+$ to obtain a positive pair $(i''',j)_+$ with $i_1<i'''$ in a $\sing$-configuration or in a $\tquad$-configuration, however both possibilities are impossible. The first possibility is a contradiction with the fact the we couldn't apply Case 1 of the construction and the second possibility is impossible since we would obtain impossible configuration \textit{(5)} of remark \ref{notpossible}.
\end{itemize}

\end{proof}

\section{Acknowledgment}

I am grateful to my master advisor Eduardo Friedman for proposing me this problem and the very valuable discussions and corrections throughout all this work.

\end{document}